\documentclass[10pt]{amsart}
\usepackage{tipa}
\usepackage{amssymb}
\usepackage{stmaryrd}
\usepackage{amsmath,amsfonts,amsthm,mathrsfs,bbm,array,subfigure}

\newtheorem{theorem}{Theorem}[section]
\newtheorem{lemma}[theorem]{Lemma}

\theoremstyle{definition}

\newtheorem{question}[theorem]{Question}

\theoremstyle{remark}

\numberwithin{equation}{section}

\begin{document}

\title[On the reducibility of some special quadrinomials]{On the reducibility of some special quadrinomials}

\author{Yong Zhang}
\address{School of Mathematics and Statistics, Changsha University of Science and Technology,
Changsha 410114, People's Republic of China}

 \email{zhangyongzju$@$163.com}

\thanks{The first author was supported by China National Natural Science Foundation Grant (No.~11501052). The second author was partly supported by China National Natural Science
Foundation Grant (No. 11501477), the Fundamental Research Funds for
the Central Universities and the Science Fund of Fujian Province
(No. 2015J01024).}

\author{Huilin Zhu}
\address{School of Mathematical Sciences, Xiamen University, Xiamen 361005, People's Republic of China}

 \email{hlzhu$@$xmu.edu.cn}

\subjclass[2010]{Primary 13P05; Secondary 11G30, 11G05}

\date{}

\keywords{quadrinomial, elliptic curve, hyperelliptic curve,
Chabauty's method.}

\begin{abstract}
Let $n>m>k$ be positive integers and let $a,b,c$ be nonzero rational
numbers. We consider the reducibility of some special quadrinomials
$x^n+ax^m+bx^k+c$ with $n=4$ and 5, which related to the study of
rational points on certain elliptic curves or hyperelliptic curves.

\end{abstract}

\maketitle

\section{Introduction}

Let $f(x)$ be trinomial or quadrinomial with $deg(f)\geq3$ and
rational coefficients, there are many authors investigated the
factorizations of $f(x)$, i.e.,
\begin{equation}
f(x)=f_1(x)f_2(x)\cdots f_k(x),
\end{equation}
where $f_i(x) \in \mathbb{Q}[x]$ with $1\leq deg(f_i)<deg(f)$,
$i=1,\cdots,k$. We can refer to
\cite{BremnerUlas1,BremnerUlas2,Fried-Schinzel,Hajdu-Tijdeman,Jankauskas,Ljunggren,Schinzel-1993,Schinzel-2000,Schinzel-2001,Schinzel-2007}.

A polynomial $f(x)$ with rational coefficients is \emph{primitive
reducible} if it is reducible but $f(x^{1/l})$ is not reducible for
any integer $l\geq2$. At the West Coast Number Theory conference in
2007, P.G. Walsh \cite{West} posed ``Is there a primitive reducible
polynomial of the form $x^i + x^j + x^k + 4$ with $0 < k < j < i$
and $i > 17$?" and ``Is there a primitive reducible polynomial of
the form $x^i + x^j + x^k + a$ with $0 < k < j < i$, $a\in
\mathbb{Z}$ and $a > 4$, and not divisible by a linear or quadratic
polynomial?"

Let $f(a,x)=x^n+x^m+x^k+a,n>m>k\geq1,a\in\mathbb{Q}$. In 2010, J.
Jankauskas \cite{Jankauskas} proved the only primitive quadrinomial
$f(4,x)$, such that $f(4,x^l)$ is reducible for some $l>1$, is $x^4
+ x^3 + x^2 + 4$. He also obtained some examples of reducible
quadrinomial $x^n+x^m+x^k+a$ with $a\in \mathbb{Z}$ and $a<-5$, such
that all the irreducible factors of $f(a,x)$ are of degree $\geq3.$
In 2015, A. Bremner and M. Ulas \cite{BremnerUlas2} studied the
reducible of quadrinomials $f(a,x)$ with $4\leq n\leq 6$ in a more
systematic way and gave further examples of reducible $f(a, x), a
\in \mathbb{Q},$ such that all the irreducible factors with degree
$\geq 3$.

Now we investigate the divisibility of the following quadrinomial
\[f_{n,m,k}(a,b,c,x)=x^n+ax^m+bx^k+c,n>m>k\geq1,n>3,\]
by the quadratic polynomial $x^2+px+q,p^2+q^2\neq0$, where $a,b,c$
are nonzero rational numbers and $p,q$ are rational numbers. In
fact, M. Fried and A. Schinzel \cite{Fried-Schinzel} studied the
(ir)reducibility of general quadrinomials in 1972. Here, we care
about the cases with $a,b,c$ are special forms such that this
problem can be reduced to the study of rational points on certain
elliptic curves or hyperelliptic curves with genus 1 or 2. Noting
that the reducibility of $f_{n,m,k}(a,b,c,x)$ is invariant under the
transformations $f(x)\mapsto x^nf(1/x)$ and $f(x) \mapsto f(x)/c$,
we just consider the cases
\[x^n+ax^m+x^k+1,x^n+ax^m+ax^k+1,x^n+ax^m+x^k+a,x^n+x^m+ax^k+a\]
of $f_{n,m,k}(a,b,c,x)$ for $n=4$ and 5.

\section{Reducible quadrinomials with degree 4}

In this section we study the reducibility of the quadrinomial
$f_{4,m,k}(a,b,c,x)=x^4+ax^m+bx^k+c$ for four special cases
$(a,b,c)=(a,1,1),(a,a,1),(a,1,a),(1,a,a).$ For the briefness of the
discussion, we introduce the following two lemmas.

\begin{lemma}\label{Lem2.1} $($\cite[Lemma 3.1]{BremnerUlas2}$)$ Let $n\in \mathbb{N}$. Then we have
\[x^n \mod (x^2 + px + q) = A_n(p, q)x + B_n(p, q),\]
where $A_0(p, q) = 0, A_1(p, q) = 1, B_0(p, q) = 1, B_1(p, q) = 0,$
and where for $n\geq2$:
\[\begin{split}
&A_n(p, q) = -pA_{n-1}(p, q) -qA_{n-2}(p, q),\\
&B_n(p, q) = -pB_{n-1}(p, q)-qB_{n-2}(p, q).
\end{split}\]
\end{lemma}

By the Corollary 3.2 of \cite{BremnerUlas2}, we have

\begin{lemma}\label{Lem2.2} If $f_{n,m,k}(a,b,c,x)$ is divisible by $x^2 + px + q$ for
some $a,b,c,p,q$, then
\begin{equation}\label{Eq21}
\begin{split}
&A_n(p, q) +aA_m(p, q)+bA_k(p, q)=0,\\
&c=-B_n(p, q)-aB_m(p, q)-bB_k(p, q).
\end{split}
\end{equation}
\end{lemma}

By the theory of elliptic curves, we have the following theorems.

\begin{theorem}\label{Thm2.3} Let $f_{4,m,k}(a,x)=x^4+ax^m+x^k+1$, where $4>m>k\geq1$ and $a\neq0.$\\

$(1)$ If $(m,k)=(2,1)$ then $f_{4,2,1}(a,x)$ is divisible by
$x^2+px+q$ if and only if
\[a=\frac{q^6-q^4-q^3-q^2+1}{q(q-1)^2(q+1)^2},p=\frac{-q}{q^2-1},\]where
$q\neq0,\pm1.$ In this case we have
\[f_{4,2,1}\bigg(\frac{q^6-q^4-q^3-q^2+1}{q(q-1)^2(q+1)^2},x\bigg)=\bigg(x^2-\frac{q}{q^2-1}x+q\bigg)\bigg(x^2+\frac{q^2}{q^2-1}x+\frac{1}{q}\bigg).\]

$(2)$ If $(m,k)=(3,1)$ then $f_{4,3,1}(a,x)$ is divisible by
$x^2+px+q$ if and only if
\[a=1,p=2,q=1;~or~a=1,p=-1,q=1.\]
In this case we have
\[f_{4,3,1}(1,x)=(x^2-x+1)(x+1)^2.\]

$(3)$ If $(m,k)=(3,2)$ then $f_{4,3,2}(a,x)$ does not have a
quadratic factor.
\end{theorem}

\begin{proof}[\textbf{Proof of Theorem 2.3.}]

(1) Case $(m,k)=(2,1)$. From Eq. (\ref{Eq21}), we have
\[-p^3-ap+2pq+1=0,-p^2q-aq+q^2+1=0.\]
Solve the above two Diophantine equations with respect to $a$, we
get
\[-\frac{p^3-2pq-1}{p}=-\frac{p^2q-q^2-1}{q}.\]
Then \[p=\frac{-q}{q^2-1}.\] Hence,
\[a=\frac{q^6-q^4-q^3-q^2+1}{q(q-1)^2(q+1)^2}.\]

(2) Case $(m,k)=(3,1)$. From Eq. (\ref{Eq21}), we have
\[ap^2-p^3-aq+2pq+1=0,apq-p^2q+q^2+1=0.\]
Solve the above two Diophantine equations with respect to $a$, we
get
\[\frac{p^3-2pq-1}{p^2-q}=\frac{p^2q-q^2-1}{pq}.\]
Then \[p=\frac{q}{2}\pm\frac{\sqrt{4q^3+q^2+4q}}{2}.\] To make $p$
be a rational number, set $r^2=4q^3+q^2+4q$. This is an elliptic
curve equivalent to $Y^2=X^3+X^2+16X$, the rank of which is 0 and
the torsion points are $(X,Y)=(0,0),(4;\pm12).$ (For the torsion
points, we omit the point $\mathcal{O}$ at infinity here and in the
following.) Hence, $(p,q)=(0,0),(-1,2;1)$, which lead to $a=1$.

(3) Case $(m,k)=(3,2)$. From Eq. (\ref{Eq21}), we have
\[ap^2-p^3-aq-p+2pq=0,apq-p^2q-q+q^2+1=0.\]
Solve the above two Diophantine equations with respect to $a$, we
get
\[\frac{(p^2-2q+1)p}{p^2-q}=\frac{p^2q-q^2+q+1}{pq}.\]
Then \[p=\pm\sqrt{q(q^2-q+1)}.\] Let $r^2=q(q^2-q+1)$, then $p$ will
be a rational number. The above curve is an elliptic curve of rank 0
with trivial torsion points $(r,q)=(0,0),(\pm1;1)$. Hence,
$(p,q)=(0,0),(\pm1;1)$, which lead to $a=0$.
\end{proof}

\begin{theorem}\label{Thm2.4} Let $f_{4,m,k}(a,a,1,x)=x^4+ax^m+ax^k+1$, where $4>m>k\geq1$ and $a\neq0.$\\

$(1)$ If $(m,k)=(2,1)$ then the set, say $\mathcal{A}$, of those $a$
such that $f_{4,2,1}(a,a,1,x)$ is divisible by $x^2+px+q$ is
infinite. More precisely the set $\mathcal{A}$ is parameterized by
the rational points on the rank one elliptic curve
\[E_1:~Y^2= X^3+4X^2+16X+64.\]

$(2)$ If $(m,k)=(3,1)$ then $f_{4,3,1}(a,a,1,x)$ is divisible by
$x^2+px+q$ if and only if
\[a=\frac{p^2-2}{p},q=1,\]where
$p\neq0.$ In this case we have
\[f_{4,3,1}\bigg(\frac{p^2-2}{p},\frac{p^2-2}{p},1,x\bigg)=(x^2+px+1)(x^2-\frac{2}{p}x+1).\]

$(3)$ If $(m,k)=(3,2)$ it is the case (1) by the transformation
$f(x)\mapsto x^4f(1/x)$.
\end{theorem}

\begin{proof}[\textbf{Proof of Theorem 2.4.}]
(1) Case $(m,k)=(2,1)$. By Eq. (\ref{Eq21}), we have
\[-p^3-ap+2pq+a=0,-p^2q-aq+q^2+1=0.\]
Solve the above two Diophantine equations with respect to $a$, we
get
\[-\frac{p(p^2-2q)}{p-1}=-\frac{p^2q-q^2-1}{q}.\]
This leads to \[p^2q-pq^2-q^2+p-1=0.\] Solve it for $q$, we have
\[q=\frac{p^2+\sqrt{p^4+4p^2-4}}{2(p+1)}.\]If $q$ is a rational
number, it needs $r^2=p^4+4p^2-4.$ This curve is equivalent to the
elliptic curve $E_1:~Y^2= X^3+4X^2+16X+64,$ the rank of which is 1
and there are infinitely many rational points on it. Hence, there
are infinitely many $(p,q)$, leading to infinitely many $a$.

(2) Case $(m,k)=(3,1)$. By Eq. (\ref{Eq21}), we have
\[ap^2-p^3-aq+2pq+a=0,apq-p^2q+q^2+1=0.\]
Solve the above two Diophantine equations with respect to $a$, we
get
\[\frac{p(p^2-2q)}{p^2-q+1}=\frac{p^2q-q^2-1}{pq}.\]
This leads to  \[(q-1)(p^2+q^2+1)=0.\] Then $q=1$. Hence,
\[a=\frac{p^2-2}{p}.\]
The result follows.
\end{proof}

\begin{theorem}\label{Thm2.5} Let $f_{4,m,k}(a,1,a,x)=x^4+ax^m+x^k+a$, where $4>m>k\geq1$ and $a\neq0.$\\

$(1)$ If $(m,k)=(2,1)$ then the set, say $\mathcal{A}$, of those $a$
such that $f_{4,2,1}(a,1,a,x)$ is divisible by $x^2+px+q$ is
infinite. More precisely the set $\mathcal{A}$ is parameterized by
the rational points on the rank one elliptic curve
\[E_1:~Y^2= X^3+4X^2+16X+64.\]

$(2)$ If $(m,k)=(3,1)$ then $f_{4,3,1}(a,1,a,x)$ is divisible by
$x^2+px+q$ if and only if
\[a=p-1,q=p-1,\]where
$p\neq0,1.$ In this case we have
\[f_{4,3,1}(p-1,1,p-1,x)=(x^2-x+1)(x^2+px+p-1).\]

$(3)$ If $(m,k)=(3,2)$ it is the case (1) by the transformations
$f(x)\mapsto x^4f(1/x)$ and $f(x) \mapsto f(x)/a$.
\end{theorem}

\begin{proof}[\textbf{Proof of Theorem 2.5.}]
(1) Case $(m,k)=(2,1)$. From Eq. (\ref{Eq21}), we have
\[-p^3-ap+2pq+1=0,-p^2q-aq+q^2+a=0.\]
Solve the above two Diophantine equations with respect to $a$, we
get
\[-\frac{p^3-2pq-1}{p}=-\frac{q(p^2-q)}{q-1}.\]
This leads to \[p^3+pq^2-2pq+q-1=0.\] Solve it for $q$, we have
\[q=\frac{2p-1+\sqrt{-4p^4+4p^2+1}}{2p}.\]If $q$ is a rational
number, it needs $r^2=-4p^4+4p^2+1.$ Let
$u=\frac{r}{p^2},v=\frac{1}{p}$, then $u^2=v^4+4v^2-4$. The
remainder proof is similar to the case (1) of Theorem 2.4.

(2) Case $(m,k)=(3,1)$. From Eq. (\ref{Eq21}), we have
\[ap^2-p^3-aq+2pq+1=0,apq-p^2q+q^2+a=0.\]
Solve the above two Diophantine equations with respect to $a$, we
get
\[\frac{p^3-2pq-1}{p^2-q}=\frac{q(p^2-q)}{pq+1}.\]
This leads to  \[(-q-1+p)(p^2+pq+q^2+p-q+1)=0.\] Then $q=p-1$.
Hence,
\[a=p-1.\]
The result follows.
\end{proof}

\begin{theorem}\label{Thm2.6} Let $f_{4,m,k}(1,a,a,x)=x^4+x^m+ax^k+a$, where $4>m>k\geq1$ and $a\neq0.$\\

$(1)$ If $(m,k)=(2,1)$ then $f_{4,2,1}(1,a,a,x)$ does not have a
quadratic factor.

$(2)$ If $(m,k)=(3,1)$ then $f_{4,3,1}(1,a,a,x)$ is divisible by
$x^2+px+q$ if and only if
\[a=-p^3,q=p^2;~or~a=(p-1)^3,q=p-1,\]where
$p\neq0,1.$ In this case we have
\[f_{4,3,1}(1,-p^3,-p^3,x)=(x^2+px+p^2)(x^2-(p-1)x-p);\]or
\[f_{4,3,1}(1,(p-1)^3,(p-1)^3,x)=(x^2+px+p-1)(x^2-(p-1)x+p^2-2p+1).\]

$(3)$ If $(m,k)=(3,2)$ it is the case (1) by the transformations
$f(x)\mapsto x^4f(1/x)$ and $f(x) \mapsto f(x)/a$.
\end{theorem}

\begin{proof}[\textbf{Proof of Theorem 2.6.}]
(1) Case $(m,k)=(2,1)$. By Eq. (\ref{Eq21}), we have
\[-p^3+2pq+a-p=0,-p^2q+q^2+a-q=0.\]
Solve the above two Diophantine equations with respect to $a$, we
get
\[p^3-2pq+p=p^2q-q^2+q.\]
This leads to \[p^3-p^2q-2pq+q^2+p-q=0.\] Solve it for $q$, we have
\[q=\frac{p^2+2p+1+\sqrt{p^4+6p^2+1}}{2}.\]If $q$ is a rational
number, it needs $r^2=p^4+6p^2+1.$ This curve is equivalent to the
elliptic curve $Y^2= X^3+6X^2-4X-24,$ the rank of which is 0 and the
torsion points are $(X,Y)=(-6,0),(\pm2,0)$. Hence, $p=0,q=0,1$, and
$a=0$.

(2) Case $(m,k)=(3,1)$. By Eq. (\ref{Eq21}), we have
\[-p^3+p^2+2pq+a-q=0,-p^2q+pq+q^2+a=0.\]
Solve the above two Diophantine equations with respect to $a$, we
get
\[p^3-p^2-2pq+q=p^2q-pq-q^2.\]
This leads to  \[(-q-1+p)(p^2-q)=0.\] Then $q=p^2$ or $p-1$. Hence,
\[a=-p^3~or~(p-1)^3.\]
The result follows.
\end{proof}

\section{Reducible quadrinomials with degree 5}

Now we investigate the reducibility of the quadrinomial
$f_{5,m,k}(a,b,c,x)=x^5+ax^m+bx^k+c$ for three special cases
$(a,b,c)=(a,1,1),(a,a,1),(a,1,a),(1,a,a).$ In the following
theorems, the cases without asterisk are absolute, and the cases
with asterisk represent the results are conjectural.

\begin{theorem}\label{Thm3.1} Let $f_{5,m,k}(a,x)=x^5+ax^m+x^k+1$, where $5>m>k\geq1$ and $a\neq0.$\\

$(1)$ If $(m,k)=(2,1)$ then $f_{5,2,1}(a,x)$ is divisible by
$x^2+px+q$ if and only if
\[a=-3,p=-2,q=1.\]
In this case we have
\[f_{5,2,1}(-3,x)=(x-1)^2(x^3+2x^2+3x+1).\]

$(2)$ If $(m,k)=(3,1)$ then $f_{5,3,1}(a,x)$ is divisible by
$x^2+px+q$ if and only if
\[a=2,p=-1,q=1.\]
In this case we have
\[f_{5,2,1}(2,x)=(x^2-x+1)(x^3+x^2+2x+1).\]

$(3)$ If $(m,k)=(3,2)$ then $f_{5,3,2}(a,x)$ is divisible by
$x^2+px+q$ if and only if
\[\begin{split}
&a=1,p=-1,q=1;~or~a=1,p=0,q=1;\\
or~&a=-\frac{19397}{1458},p=\frac{10}{27},q=\frac{1}{6};~or~a=\frac{2597}{192},p=-\frac{3}{8},q=\frac{1}{6}.
\end{split}\] In this case we have
\[\begin{split}
&f_{5,3,2}(1,x)=(x+1)(x^2+1)(x^2-x+1);\\
&or~f_{5,3,2}\bigg(-\frac{19397}{1458},x\bigg)=\bigg(x^2+\frac{10}{27}x+\frac{1}{6}\bigg)\bigg(x^3-\frac{10}{27}x^2-\frac{40}{3}x+6\bigg);\\
&or~f_{5,3,2}\bigg(\frac{2597}{192},x\bigg)=\bigg(x^2-\frac{3}{8}x+\frac{1}{6}\bigg)\bigg(x^3+\frac{3}{8}x^2+\frac{27}{2}x+6\bigg).
\end{split}\]

*$(4)$ If $(m,k)=(4,1)$ then $f_{5,4,1}(a,x)$ is
divisible by $x^2+px+q$ if and only if
\[a=-2,p=-1,q=1;~or~a=-\frac{1055}{16},p=\frac{1}{16},q=\frac{1}{8}.\]
In this case we have
\[\begin{split}
&f_{5,4,1}(-2,x)=(x^2-x-1)(x^3-x^2-1);\\
&or~f_{5,4,1}\bigg(-\frac{1055}{16},x\bigg)=\bigg(x^2+\frac{1}{16}x+\frac{1}{8}\bigg)(x^3-66x^2+4x+8).\end{split}\]

*$(5)$ If $(m,k)=(4,2)$ then $f_{5,4,2}(a,x)$ does not have a
quadratic factor.

$(6)$ If $(m,k)=(4,3)$ then $f_{5,4,3}(a,x)$ is divisible by
$x^2+px+q$ if and only if
\[a=-1,p=0,q=1.\]
In this case we have
\[f_{5,4,1}(-1,x)=(x^2+1)(x^3-x^2+1).\]
\end{theorem}

\begin{proof}[\textbf{Proof of Theorem 3.1.}]
(1) Case $(m,k)=(2,1)$. From Eq. (\ref{Eq21}), we have
\[p^4-3p^2q-ap+q^2+1=0,p^3q-2pq^2-aq+1=0.\]
Solve the above two Diophantine equations with respect to $a$, we
get
\[\frac{p^4-3p^2q+q^2+1}{p}=\frac{p^3q-2pq^2+1}{q}.\]This leads to
\[-\frac{p^2q^2-q^3+p-q}{pq}=0.\]Solve it for $p$, we have
\[p=\frac{-1\pm\sqrt{4q^5+4q^3+1}}{2q^2}.\] If $p$ is a rational number,
it needs $r^2=4q^5+4q^3+1.$ This is a hyperelliptic quintic curve of
genus 2. The rank of the Jacobian variety is 1, and Magma's Chabauty
routines \cite{Chabauty-Magma} determine the only finite rational
points are $(r,q)=(\pm1;0),(\pm3;1)$, which lead to
$(p,q)=(-2,1),(1,1),(0,0)$. Hence, $a=-3$.

(2) Case $(m,k)=(3,1)$. From Eq. (\ref{Eq21}), we have
\[p^4+ap^2-3p^2q-aq+q^2+1=0,p^3q+apq-2pq^2+1=0.\]
Solve the above two Diophantine equations with respect to $a$, we
get
\[-\frac{p^4-3p^2q+q^2+1}{p^2-q}=-\frac{p^3q-2pq^2+1}{pq}.\]
Then
\[\frac{pq^3+p^2-pq-q}{(p^2-q)pq}=0.\]Solve it for $p$, we have
\[p=\frac{-q^3+q\pm\sqrt{q^6-2q^4+q^2+4q}}{2}.\]
Let $r^2=q^6-2q^4+q^2+4q,$ then $p$ is a rational number. This is a
hyperelliptic sextic curve of genus 2. The rank of the Jacobian
variety is 1, and Magma's Chabauty routines determine the only
finite rational points are $(r,q)=(0,0),(\pm2;1)$, which lead to
$(p,q)=(\pm1;1),(0,0)$. Hence, $a=2$.

(3) Case $(m,k)=(3,2)$. From Eq. (\ref{Eq21}), we have
\[p^4+ap^2-3p^2q-aq-p+q^2=0,p^3q+apq-2pq^2-q+1=0.\]
Solve the above two Diophantine equations with respect to $a$, we
get
\[-\frac{p^4-3p^2q+q^2-p}{p^2-q}=-\frac{p^3q-2pq^2-q+1}{pq}.\]
This leads to
\[\frac{pq^3+p^2+q^2-q}{(p^2-q)pq}=0.\]Solve it for $p$, we have
\[p=\frac{-q^3\pm\sqrt{q^6-4q^2+4q}}{2}.\] If $p$ is a rational number,
it needs $r^2=q^6-4q^2+4q.$ This is a hyperelliptic sextic curve of
genus 2. The rank of the Jacobian variety is 1, and Magma's Chabauty
routines determine the only finite rational points are
\[(r,q)=(0,0),(\pm1;1),\bigg(\pm\frac{161}{6};\frac{1}{6}\bigg),\]
which lead to
\[(p,q)=(0,1),(-1,1),\bigg(\frac{10}{27},\frac{1}{6}\bigg),\bigg(-\frac{3}{8},\frac{1}{6}\bigg).\]
Hence, \[a=1,-\frac{19397}{1458},\frac{2597}{192}.\]

*(4) Case $(m,k)=(4,1)$. From Eq. (\ref{Eq21}), we have
\[-ap^3+p^4+2apq-3p^2q+q^2+1=0,-ap^2q+p^3q+aq^2-2pq^2+1=0.\]
Solve the above two Diophantine equations with respect to $a$, we
get
\[\frac{p^4-3p^2q+q^2+1}{p(p^2-2q)}=\frac{p^3q-2pq^2+1}{q(p^2-q)}.\]
Then
\[-\frac{q^4+p^3-p^2q-2pq+q^2}{p(p^2-2q)q(p^2-q)}=0.\]It
needs to consider $q^4+p^3-p^2q-2pq+q^2=0$. Let $p=tq$, we get
\[q^2(qt^3-qt^2+q^2-2t+1)=0.\]Solve it for $q$, we get
\[q=\frac{-t^3+t^2\pm\sqrt{t^6-2t^5+t^4+8t-4}}{2}.\] Let $r^2=t^6-2t^5+t^4+8t-4$, then $q$ is a rational number.
This is a hyperelliptic sextic curve of genus 2. The rank of the
Jacobian variety is 2, so standard Chabauty arguments do not apply
and we are unable to determine explicitly all the rational points.
But the obviously rational points are
\[(r,q)=(\pm2;1),\bigg(\pm\frac{1}{8};\frac{1}{2}\bigg),\] which lead to
\[(p,q)=(0,0),(-1,-1),(1,1),\bigg(\frac{1}{16},\frac{1}{8}\bigg).\]
Hence, \[a=-2,-\frac{1055}{16}.\]

*(5) Case $(m,k)=(4,2)$. From Eq. (\ref{Eq21}), we have
\[-ap^3+p^4+2apq-3p^2q-p+q^2=0,-ap^2q+p^3q+aq^2-2pq^2-q+1=0.\]
Solve the above two Diophantine equations with respect to $a$, we
get
\[\frac{p^4-3p^2q+q^2-p}{p(p^2-2q)}=\frac{p^3q-2pq^2-q+1}{q(p^2-q)}.\]
This leads to
\[-\frac{q^4+p^3+pq^2-2pq}{p(p^2-2q)q(p^2-q)}=0.\]It
needs to consider $q^4+p^3+pq^2-2pq=0$. Let $p=tq$, we get
\[q^2(qt^3+q^2+qt-2t)=0.\]Solve it for $q$, we get
\[q=\frac{-t^3-t\pm\sqrt{t^6+2t^4+t^2+8t}}{2}.\] If $q$ is a rational number,
it needs $r^2=t^6+2t^4+t^2+8t.$ This is a hyperelliptic sextic curve
of genus 2. The rank of the Jacobian variety is 1, but we cannot
find the generator of its Jacobian, and Magma's Chabauty routines
don't determine the complete rational point. But it is believed that
the only rational point is $(r,t)=(0,0)$, which leads to
$(p,q)=(0,0)$. Hence, the result follows.

(6) Case $(m,k)=(4,3)$. From Eq. (\ref{Eq21}), we have
\[-ap^3+p^4+2apq+p^2-3p^2q-q+q^2=0,-ap^2q+p^3q+aq^2+pq-2pq^2+1=0.\]
Solve the above two Diophantine equations with respect to $a$, we
get
\[\frac{p^4-3p^2q+p^2+q^2-q}{p(p^2-2q)}=\frac{p^3q-2pq^2+pq+1}{q(p^2-q)}.\]
Then
\[-\frac{q^4+p^3-q^3-2pq}{p(p^2-2q)q(p^2-q)}=0.\]It
needs to consider $q^4+p^3-q^3-2pq=0$. Let $p=tq$, we get
\[q^2(qt^3+q^2-q-2t)=0.\]Solve it for $q$, we get
\[q=\frac{-t^3+1\pm\sqrt{t^6-2t^3+8t+1}}{2}.\] Let $r^2=t^6-2t^3+8t+1,$ then $q$ is a rational number. This is
a hyperelliptic sextic curve of genus 2. The rank of the Jacobian
variety is 1, and Magma's Chabauty routines determine the only
finite rational points are $(r,t)=(\pm1;1)$, which lead to
$(p,q)=(0;0,1)$. Hence, $a=-1$.
\end{proof}

\begin{theorem}\label{Thm3.2} Let $f_{5,m,k}(a,a,1,x)=x^5+ax^m+ax^k+1$, where $5>m>k\geq1$ and $a\neq0.$\\

$(1)$ If $(m,k)=(2,1)$ then $f_{5,2,1}(a,a,1,x)$ is divisible by
$x^2+px+q$ if and only if
\[a=\frac{q^4+q^3+q^2+q+1}{q},p=q+1,\]where $q\neq0$.
In this case we have
\[\begin{split}f_{5,2,1}&\bigg(\frac{q^4+q^3+q^2+q+1}{q},\frac{q^4+q^3+q^2+q+1}{q},1,x\bigg)\\
&=(x^2+(q+1)x+q)\big(x^3+(-q-1)x^2+(q^2+q+1)x+\frac{1}{q}\big).\end{split}\]

*$(2)$ If $(m,k)=(3,1)$ then $f_{5,3,1}(a,a,1,x)$ does not have a
quadratic factor.

$(3)$ If $(m,k)=(3,2)$ then $f_{5,3,2}(a,a,1,x)$ is divisible by
$x^2+px+q$ if and only if
\[a=-\frac{q^4+q^3+q^2+q+1}{q^2},p=q+1;\]
or\[a=\frac{q^4+q^3+q^2+q+1}{q(q^2+2q+1)},p=-\frac{q}{q+1},\]where
$q\neq0,-1.$ In this case we have
\[\begin{split}f_{5,3,2}&\bigg(-\frac{q^4+q^3+q^2+q+1}{q^2},-\frac{q^4+q^3+q^2+q+1}{q^2},1,x\bigg)\\
&=(x^2+(q+1)x+q)\bigg(x^3-(q+1)x^2-\frac{(1+q)x}{q^2}+\frac{1}{q}\bigg);\end{split}\]
or
\[\begin{split}f_{5,3,2}&\bigg(\frac{q^4+q^3+q^2+q+1}{q(q^2+2q+1)},\frac{q^4+q^3+q^2+q+1}{q(q^2+2q+1)},1,x\bigg)\\
&=\bigg(x^2-\frac{qx}{q+1}+q\bigg)\bigg(x^3+\frac{qx^2}{q+1}+\frac{x}{q(q+1)}+\frac{1}{q}\bigg).\end{split}\]

$(4)$ If $(m,k)=(4,1)$ then $f_{5,4,1}(a,a,1,x)$ is divisible by
$x^2+px+q$ if and only if
\[a=\frac{q^4+q^3+q^2+q+1}{q(q^2+q+1)},p=q+1.\]
In this case we have
\[\begin{split}
f_{5,4,1}&\bigg(\frac{q^4+q^3+q^2+q+1}{q(q^2+q+1)},\frac{q^4+q^3+q^2+q+1}{q(q^2+q+1)},1,x\bigg)\\
&=(x^2+(q+1)x+q)\bigg(x^3+\frac{-q^3-q^2+1}{q(q^2+q+1)}x^2+\frac{(q^3-q-1)}{q(q^2+q+1)}x+\frac{1}{q}\bigg).\end{split}\]

$(5)$ If $(m,k)=(4,2)$ it is the case (2) by the transformations
$f(x)\mapsto x^5f(1/x)$ and $f(x) \mapsto f(x)/a$.

$(6)$ If $(m,k)=(4,3)$ it is the case (1) by the transformations
$f(x)\mapsto x^5f(1/x)$ and $f(x) \mapsto f(x)/a$.
\end{theorem}

\begin{proof}[\textbf{Proof of Theorem 3.2.}]
(1) Case $(m,k)=(2,1)$. By Eq. (\ref{Eq21}), we have
\[p^4-3p^2q-ap+q^2+a=0,p^3q-2pq^2-aq+1=0.\]
Solve the above two Diophantine equations with respect to $a$, we
get
\[\frac{p^4-3p^2q+q^2}{p-1}=\frac{p^3q-2pq^2+1}{q}.\]This leads to
\[(-q-1+p)(p^2q+pq-q^2+q-1)=0.\]Solve it for $p$, we have
\[p=q+1;~or~\frac{-q\pm\sqrt{4q^3-3q^2+4q}}{2q}.\] If $p$ is a rational number,
it needs $r^2=4q^3-3q^2+4q.$ Let $u=4r,v=4q$, we have
$u^2=v^3-3v^2+16v$. This is an elliptic curve with rank 0 and
torsion point $(v,u)=(0,0).$ Then $q=0.$ So we get $p=q+1$ and
\[a=\frac{q^4+q^3+q^2+q+1}{q}.\]

*(2) Case $(m,k)=(3,1)$. By Eq. (\ref{Eq21}), we have
\[p^4+ap^2-3p^2q-aq+q^2+a=0,p^3q+apq-2pq^2+1=0.\]
Solve the above two Diophantine equations with respect to $a$, we
get
\[-\frac{p^4-3p^2q+q^2}{p^2-q+1}=-\frac{p^3q-2pq^2+1}{pq}.\]
Then
\[p^3q+pq^3-2pq^2+p^2-q+1=0.\]This is an equation about $p,q$ with
degree 4 and genus 3, so it's hard to get the all rational solutions
of it. But we conjecture its rational solution is $(p,q)=(0,1).$

(3) Case $(m,k)=(3,2)$. By Eq. (\ref{Eq21}), we have
\[p^4+ap^2-3p^2q-ap-aq+q^2=0,p^3q+apq-2pq^2-aq+1=0.\]
Solve the above two Diophantine equations with respect to $a$, we
get
\[-\frac{p^4-3p^2q+q^2}{p^2-p-q}=-\frac{p^3q-2pq^2+1}{q(p-1)}.\]
This leads to
\[(q-1)(pq+p+q)(-q+p-1)=0.\]Solve it for $p$, we get
\[p=q+1;~or~-\frac{q}{q+1}.\]Then
\[a=-\frac{q^4+q^3+q^2+q+1}{q^2};~or~\frac{q^4+q^3+q^2+q+1}{q(q^2+2q+1)}.\]

(4) Case $(m,k)=(4,1)$. By Eq. (\ref{Eq21}), we have
\[-ap^3+p^4+2apq-3p^2q+q^2+a=0,-ap^2q+p^3q+aq^2-2pq^2+1=0.\]
Solve the above two Diophantine equations with respect to $a$, we
get
\[\frac{p^4-3p^2q+q^2}{p^3-2pq-1}=\frac{p^3q-2pq^2+1}{q(p^2-q)}.\]
Then
\[(q-1)(-q-1+p)(p^2+pq+q^2+p+1)=0.\] Solve it for $p$, we get
\[p=q+1;~or~\frac{q+1\pm\sqrt{-3q^2+2q-3}}{2}.\] Noting that $-3q^2+2q-3<0$ for any rational number $q$, then $p=q+1$.
So
\[a=\frac{q^4+q^3+q^2+q+1}{q(q^2+q+1)}.\]
The result follows.
\end{proof}

\begin{theorem}\label{Thm3.3} Let $f_{5,m,k}(a,1,a,x)=x^5+ax^m+x^k+a$, where $5>m>k\geq1$ and $a\neq0.$\\

$(1)$ If $(m,k)=(2,1)$ then $f_{5,2,1}(a,1,a,x)$ does not have a
quadratic factor.

*$(2)$ If $(m,k)=(3,1)$ then $f_{5,3,1}(a,1,a,x)$ is divisible by
$x^2+px+q$ if and only if
\[a=\frac{1}{2},p=q=1.\] In this case we have
\[f_{5,3,2}\big(\frac{1}{2},1,\frac{1}{2},x\big)=(x^2+x+1)\big(x^3-x^2+\frac{x}{2}+\frac{1}{2}\big).\]

$(3)$ If $(m,k)=(3,2)$ then $f_{5,3,2}(a,1,a,x)$ is divisible by
$x^2+px+q$ if and only if
\[a=-q^2,p=q+1,\] where $q\neq0$. In this case we have
\[f_{5,3,2}(-q^2,1,-q^2,x)=(x^2+(q+1)x+q)(x^3-(q+1)x^2+(q+1)x-q).\]

$(4)$ If $(m,k)=(4,1)$ then $f_{5,4,1}(a,1,a,x)$ does not have a
quadratic factor.

$(5)$ If $(m,k)=(4,2)$ it is the case (2) by the transformations
$f(x)\mapsto x^5f(1/x)$ and $f(x) \mapsto f(x)/a$.

$(6)$ If $(m,k)=(4,3)$ it is the case (1) by the transformations
$f(x)\mapsto x^5f(1/x)$ and $f(x) \mapsto f(x)/a$.
\end{theorem}

\begin{proof}[\textbf{Proof of Theorem 3.3.}]
(1) Case $(m,k)=(2,1)$. From Eq. (\ref{Eq21}), we have
\[p^4-3p^2q-ap+q^2+1=0,p^3q-2pq^2-aq+a=0.\]
Solve the above two Diophantine equations with respect to $a$, we
get
\[\frac{p^4-3p^2q+q^2+1}{p}=\frac{pq(p^2-2q)}{q-1}.\]This leads to
\[p^4+p^2q^2-3p^2q-q^3+q^2-q+1=0.\]Solve it for $p$, we have
\[p=\pm\frac{\sqrt{-2q^2+6q+2\sqrt{q^4-2q^3+5q^2+4q-4}}}{2}.\] If $p$ is a rational number,
it needs $r^2=q^4-2q^3+5q^2+4q-4$ and $s^2=-2q^2+6q+2r.$ Let us
consider $r^2=q^4-2q^3+5q^2+4q-4$, this equation is equivalent to
the elliptic curve $Y^2-2XY+8Y=X^3+4X^2+16X+64$ with rank 0 and
torsion points $(X,Y)=(4,16),(-4,0).$ Then $q=\pm1,p=\pm\sqrt{2}i.$
So the result follows.

*(2) Case $(m,k)=(3,1)$. From Eq. (\ref{Eq21}), we have
\[p^4+ap^2-3p^2q-aq+q^2+1=0,p^3q+apq-2pq^2+a=0.\]
Solve the above two Diophantine equations with respect to $a$, we
get
\[-\frac{p^4-3p^2q+q^2+1}{p^2-q}=-\frac{pq(p^2-2q)}{pq+1}.\]
Then
\[p^4-pq^3-3p^2q+pq+q^2+1=0.\]This is an equation about $p,q$ with
degree 4 and genus 3, so it's hard to get the all rational solutions
of it. But we conjecture its rational solution is $(p,q)=(\pm1,1).$
Then $a=\frac{1}{2}$ with $p=q=1.$

(3) Case $(m,k)=(3,2)$. From Eq. (\ref{Eq21}), we have
\[p^4+ap^2-3p^2q-aq+q^2-p=0,p^3q+apq-2pq^2+a-q=0.\]
Solve the above two Diophantine equations with respect to $a$, we
get
\[-\frac{p^4-3p^2q+q^2-p}{p^2-q}=-\frac{q(p^3-2pq-1)}{pq+1}.\]
This leads to
\[p(-q+p-1)(p^2+pq+q^2+p-q+1)=0.\]Solve it for $p$, we get $p=q+1.$ Then $a=-q^2.$

(4) Case $(m,k)=(4,1)$. From Eq. (\ref{Eq21}), we have
\[-ap^3+p^4+2apq-3p^2q+q^2+1=0,-ap^2q+p^3q+aq^2-2pq^2+a=0.\]
Solve the above two Diophantine equations with respect to $a$, we
get
\[\frac{p^4-3p^2q+q^2+1}{p(p^2-2q)}=\frac{pq(p^2-2q)}{p^2q-q^2-1}.\]
Then
\[p^4+q^4-4p^2q+2q^2+1=0.\] Write it as
\[(p^2-2q)^2=-(q^2-1)^2,\]it is easy to see that there is no
rational solution $(p,q)$.
\end{proof}

\begin{theorem}\label{Thm3.4} Let $f_{5,m,k}(1,a,a,x)=x^5+x^m+ax^k+a$, where $5>m>k\geq1$ and $a\neq0.$\\

$(1)$ If $(m,k)=(2,1)$ then $f_{5,2,1}(1,a,a,x)$ is divisible by
$x^2+px+q$ if and only if
\[a=-(p^2-p+1)(p-1)^2,q=p-1;\]
or\[a=\frac{p(p+1)(p^2+p+1)^2}{(2p+1)^2},q=\frac{p(p^2+p+1)}{2p+1}.\] In this case we have
\[\begin{split}f_{5,3,2}&(1,-(p^2-p+1)(p-1)^2,-(p^2-p+1)(p-1)^2,x)\\
&=(x^2+px+p-1)(x^3-px^2-(-p^2+p-1)x-p^3+2p^2-2p+1);\end{split}\] or
\[\begin{split}
f_{5,3,2}&\bigg(1,\frac{p(p+1)(p^2+p+1)^2}{(2p+1)^2},\frac{p(p+1)(p^2+p+1)^2}{(2p+1)^2},x\bigg)\\
&=\bigg(x^2+px+\frac{p(p^2+p+1)}{2p+1}\bigg)\bigg(x^3-px^2+\frac{p^3-p}{2p+1}x+\frac{p^3+2p^2+2p+1}{2p+1}\bigg).\end{split}\]

$(2)$ If $(m,k)=(3,1)$ then $f_{5,3,1}(1,a,a,x)$ is divisible by
$x^2+px+q$ if and only if
\[a=\frac{125}{12},p=-\frac{1}{2},q=\frac{5}{18}.\] In this case we have
\[f_{5,3,2}\bigg(1,\frac{125}{12},\frac{125}{12},x\bigg)=\bigg(x^2-\frac{x}{2}+\frac{5}{18}\bigg)\bigg(x^3+\frac{x^2}{2}+\frac{35x}{36}+\frac{25}{72}\bigg).\]

$(3)$ If $(m,k)=(3,2)$ then $f_{5,3,2}(1,a,a,x)$ is divisible by
$x^2+px+q$ if and only if
\[a=-p^2,q=p^2,\] where $p\neq0$. In this case we have
\[f_{5,3,2}(1,-p^2,-p^2,x)=(x-p)(x^2+1)(x^2+px+p^2).\]

$(4)$ If $(m,k)=(4,1)$ then $f_{5,4,1}(1,a,a,x)$ is divisible by
$x^2+px+q$ if and only if
\[a=\frac{p^4}{2},q=\frac{p^2}{2};~or~a=-(p-1)^4,q=p-1.\] In this case we have
\[f_{5,4,1}\bigg(1,\frac{p^4}{2},\frac{p^4}{2},x\bigg)=(x+1)\bigg(x^2-px+\frac{p^2}{2}\bigg)\bigg(x^2+px+\frac{p^2}{2}\bigg);\]
or\[f_{5,4,1}(1,-(p-1)^4,-(p-1)^4,x)=(x^2+px+p-1)(x-p+1)(x^2+p^2-2p+1).\]

$(5)$ If $(m,k)=(4,2)$ it is the case (2) by the transformations
$f(x)\mapsto x^5f(1/x)$ and $f(x) \mapsto f(x)/a$.

$(6)$ If $(m,k)=(4,3)$ it is the case (1) by the transformations
$f(x)\mapsto x^5f(1/x)$ and $f(x) \mapsto f(x)/a$.
\end{theorem}

\begin{proof}[\textbf{Proof of Theorem 3.4.}]
(1) Case $(m,k)=(2,1)$. By Eq. (\ref{Eq21}), we have
\[p^4-3p^2q+q^2+a-p=0,p^3q-2pq^2+a-q=0.\]
Solve the above two Diophantine equations with respect to $a$, we
get
\[-p^4+3p^2q-q^2+p=-p^3q+2pq^2+q.\]This leads to
\[-(-q-1+p)(p^3+p^2-2pq+p-q)=0.\]Solve it for $q$, we have
\[q=p+1;~or~\frac{p(p^2+p+1)}{2p+1}.\] Then
\[a=-(p^2-p+1)(p-1)^2;~or~\frac{p(p+1)(p^2+p+1)^2}{(2p+1)^2}.\]

(2) Case $(m,k)=(3,1)$. By Eq. (\ref{Eq21}), we have
\[p^4-3p^2q+p^2+q^2+a-q=0,p^3q-2pq^2+pq+a=0.\]
Solve the above two Diophantine equations with respect to $a$, we
get
\[-p^4+3p^2q-p^2-q^2+q=-p^3q+2pq^2-pq.\]
Then
\[-p^4+p^3q+3p^2q-2pq^2-p^2+pq-q^2+q=0.\]Solve it for $q$, we get
\[q=\frac{p^3+3p^2+p+1+\sqrt{p^6-2p^5+7p^4+3p^2+2p+1}}{2p+1}.\]If $q$ is a rational number,
it needs $r^2=p^6-2p^5+7p^4+3p^2+2p+1.$ This is a hyperelliptic
sextic curve of genus 2. The rank of the Jacobian variety is 1, and
Magma's Chabauty routines determine the only finite rational points
are \[(r,p)=(\pm1;0),\bigg(\pm\frac{9}{8};-\frac{1}{2}\bigg),\]
which lead to \[(p,q)=\big(-\frac{1}{2},\frac{5}{18}\big).\] Hence,
\[a=\frac{125}{12}.\]

(3) Case $(m,k)=(3,2)$. By Eq. (\ref{Eq21}), we have
\[p^4-3p^2q-ap+p^2+q^2-q=0,p^3q-2pq^2-aq+pq+a=0.\]
Solve the above two Diophantine equations with respect to $a$, we
get
\[\frac{p^4-3p^2q+p^2+q^2-q}{p}=\frac{pq(p^2-2q+1)}{q-1}.\]
This leads to
\[(p^2+q^2-2q+1)(p^2-q)=0.\]Solve it for $q$, we get $q=p^2.$ Then $a=-p^3.$

(4) Case $(m,k)=(4,1)$. By Eq. (\ref{Eq21}), we have
\[p^4-p^3-3p^2q+2pq+q^2+a=0,p^3q-p^2q-2pq^2+q^2+a=0.\]
Solve the above two Diophantine equations with respect to $a$, we
get
\[-p^4+p^3+3p^2q-2pq-q^2=-p^3q+p^2q+2pq^2-q^2.\]
Then
\[-p(-q-1+p)(p^2-2q)=0.\] Solve it for $q$, we get
\[q=\frac{p^2}{2};~or~p-1.\] Then
\[a=\frac{p^4}{2};~or~-(p-1)^4.\]
The result follows.
\end{proof}

\section{Some related questions}
In the end, we raise some questions about the reducibility of
quadrinomials.

\begin{question} Are there infinitely many $f(x)=x^n+ax^m+bx^k+c,$ with $abc\neq0,n>m>k>0,n\geq6$ has a quadratic
factor $g(x)=x^2+px+q$ for
$(a,b,c)=(a,1,1),(a,a,1),(a,1,a),(1,a,a)$~?
\end{question}

\begin{question} If $f(x)=x^n+ax^m+bx^k+c,abc\neq0,n>m>k>0$, whether all the factors of $f(x)$ have the forms $f_i(x)=x^{n_i}+a_ix^{m_i}+b_ix^{k_i}+c_i$ where
$n_i>m_i>k_i>0,a_ib_ic_i\neq0$~?
\end{question}

\begin{question} Can we use the method of infinite descent to prove the only rational
point on the hyperelliptic curve $r^2=p^6+2p^4+p^2+8p$ is $(0,0)$?
\end{question}

\vskip20pt
\bibliographystyle{amsplain}

\end{document}